\documentclass[runningheads]{lncse}
\usepackage{silence}
\usepackage{amsmath}
\usepackage{amsfonts}
\usepackage{epsfig}
\usepackage{graphics}
\usepackage{tikz}
\usetikzlibrary{shapes,shadings}

\begin{document}
	
	\title{Stochastic B-series and order conditions for exponential integrators}
	\titlerunning{Stochastic B-series and order conditions for exponential integrators}
	
	\author{ Alemayehu Adugna Arara\inst{1} \and Kristian Debrabant\inst{2} \and Anne Kv{\ae}rn{\o}\inst{3}   }
	
	\authorrunning{A.\ A.\ Arara \and K.\ Debrabant \and A.\ Kv{\ae}rn{\o}}
	
	\institute{
		Hawassa University, School of Mathematics and Statistics, P.O. Box 05, Hawassa, Ethiopia, {\tt alemayehu\_arara@yahoo.com}
		\and University of Southern Denmark, Department of Mathematics and Computer Science, 5230 Odense M, Denmark, {\tt debrabant@imada.sdu.dk}
		\and Norwegian University of Science and Technology - NTNU, Department of
		Mathematical Sciences, NO-7491 Trondheim, Norway, {\tt anne.kvarno@ntnu.no}
	}
	
	\maketitle
	
	\begin{abstract}
We discuss stochastic differential equations with a stiff linear part and their approximation by stochastic exponential integrators. Representing the exact and approximate solutions using B-series and rooted trees, we derive the order conditions for stochastic exponential integrators. The resulting general order theory covers both It\^{o} and Stratonovich integration.
	\end{abstract}
	
	\section{Introduction}
	\label{kvarno_mini_ms18:sec:intro}
	The idea of expressing the exact and numerical solutions of different blends of
	differential equations in terms of B-series and rooted trees has been an
	indispensable tool ever since John Butcher introduced the idea in
	1963  \cite{kvarno_mini_ms18:butcher63cft}. Naturally then, such series have also been derived for
	stochastic differential equations (SDEs) by several authors, see e.g.\ \cite{kvarno_mini_ms18:debrabant08bsa} for an overview.
	
	In this paper, the focus is on $d$-dimensional SDEs of the form
	\begin{equation}\label{kvarno_mini_ms18:eq:SDE}
	\mathrm{d}X(t)=\bigg(AX(t) +
	g_0\big(X(t)\big)\bigg)\mathrm{d} t+\sum_{{m}=1}^{{M}}g_{m}(X(t))\star\mathrm{d} W_m(t),\quad X(0)=x_{0},
	\end{equation}
	or in integral form
	\begin{equation}\label{kvarno_mini_ms18:eq:SDEexpform}
	  X(t) = \mathrm{e}^{tA}x_0 +  \int_{0}^t \mathrm{e}^{(t-s)A}g_0(X(s)) \mathrm{d} s +  \sum_{{m}=1}^{{M}} \int_{0}^t
	\mathrm{e}^{(t-s)A}g_m(X(s))\star\mathrm{d} W_m(s),
	\end{equation}
	in which case the linear term $AX(t)$, $A\in \mathbb{R}^{d\times d}$ constant will be treated with particular care by the use of exponential
	integrators, see e.g.
	\cite{kvarno_mini_ms18:becker16aew,kvarno_mini_ms18:cohen13atm,kvarno_mini_ms18:tambue16wcf} and
	references
	therein. The integrals w.\,r.\,t.\ the components of the ${M}$-dimensional Wiener process $W(t)$ can be
	interpreted e.\,g.\ as an It\^{o} or a Stratonovich integral. The coefficients $g_m\,:\,\mathbb{R}^d
	\rightarrow \mathbb{R}^d$ are sufficiently differentiable and satisfy a Lipschitz and a linear growth
	condition. For Stratonovich SDEs, we require in addition that $g_m$ are differentiable and that
	also the $g_m' g_m $ satisfy a Lipschitz and a linear growth condition.
	In the following, we will denote $\mathrm{d} t = \mathrm{d} W_0(t)$.

	For the numerical solution of \eqref{kvarno_mini_ms18:eq:SDE} we consider a general class of $\nu$-stage stochastic exponential integrators:
	\begin{subequations} \label{kvarno_mini_ms18:eq:Btab}
		\begin{align}
		H_i &= \mathrm{e}^{c_ihA}Y_n + \sum_{{m}=0}^{{M}} \sum_{j=1}^\nu
		Z^{(m)}_{ij}(A)\cdot g_{{m}}(H_j), \qquad i = 1, \ldots, \nu,\\
		Y_{n+1} &= \mathrm{e}^{hA} Y_n +  \sum_{{m}=0}^{{M}} \sum_{i=1}^\nu
		z_i^{({m})}(A) \cdot g_{{m}}(H_i),
		\end{align}
	\end{subequations}
	where typically the coefficients $Z_{ij}^{({m})}$ and $z_{i}^{({m})}$ are random variables depending on the stepsize $h$, the matrix $A$ and the Wiener processes.
	
	\begin{example}
		\label{kvarno_mini_ms18:ex:SETDRK}
		A 2-stage stochastic exponential time-differencing Runge--Kutta method (SETDRK) for ${M}=1$ is given by:
		\begin{align*}
		H_1 & = Y_n, \\
		H_2 & = Y_n + \sqrt{h}g_1(H_1),\\
		Y_{n+1} & = \mathrm{e}^{hA}Y_n +\int_{t_n}^{t_{n+1}}\mathrm{e}^{(t_{n+1}-s)A} \mathrm{d} s\cdot g_0(H_1)\\& + \int_{t_n}^{t_{n+1}}\mathrm{e}^{(t_{n+1}-s)A}\star\mathrm{d} W_1(s)\cdot g_1(H_1)\\
		& +
		\frac{1}{\sqrt{h}}\int_{t_n}^{t_{n+1}}\mathrm{e}^{(t_{n+1}-s)A}\; W_1(s)\star\mathrm{d} W_1(s)\cdot\left(-g_1(H_1)+g_1(H_2)\right),
		\end{align*}
		where $t_{n+1} = t_n+h$.
	\end{example}
	
	In the following, the ideas introduced in \cite{kvarno_mini_ms18:debrabant08bsa} will be used to derive a B-series representation of \eqref{kvarno_mini_ms18:eq:SDEexpform} and corresponding order conditions for the method \eqref{kvarno_mini_ms18:eq:Btab}.
	In the deterministic case, such analysis  has been carried out in \cite{kvarno_mini_ms18:berland05bsa,kvarno_mini_ms18:hochbruck05eer}.

	\section{B-series and order conditions for exponential integrators}
	\label{kvarno_mini_ms18:sec:bseries}
	To develop B-series for the exact solution of \eqref{kvarno_mini_ms18:eq:SDE} and one step numerical exponential integrators of the form \eqref{kvarno_mini_ms18:eq:Btab}, we use the following definitions of the trees associated to the stochastic differential equation \eqref{kvarno_mini_ms18:eq:SDE} and their corresponding elementary differentials.
	\begin{definition}[trees]
		The set of ${M} +2$-colored, rooted trees
		\[ T = \{\emptyset\}\cup T_0 \cup T_1 \cup \cdots \cup T_{M} \cup T_A \]
		is recursively defined as follows:
		\begin{enumerate}
			\item The graph $\bullet_{m} = [\emptyset]_{m}$ with only one vertex of color ${m}$ belongs to $T_{m}$, and $\bullet_A = [\emptyset]_A$ with only one vertex of color $A$ belongs to $T_A$,
			\item[2.] Let $\tau = [\tau_1, \tau_2, \ldots, \tau_\kappa]_{m}$ be the tree formed by joining the subtrees $\tau_1, \tau_2, \ldots, \tau_\kappa$ each by a single branch to a common root of color ${m}$ and $\tau = [\tau_1]_A$ be the tree formed by joining the subtree $\tau_1$ to a root of color $A$. If $\tau_1, \tau_2, \ldots, \tau_\kappa\in T$, then $\tau = [\tau_1, \tau_2, \ldots, \tau_\kappa]_{m}\in T_{m}$ and $[\tau_1]_A\in T_A$,
		\end{enumerate}
		for $m=0,\dots,M$.
	\end{definition}
	\begin{definition}[elementary differential]
		For a tree $\tau\in T$ the elementary differential is a mapping $F(\tau):\mathbb{R}^d\to \mathbb{R}^d$ defined recursively by
		\begin{enumerate}
			\item $F(\emptyset)(x_0) = x_0,$
			\item $F(\bullet_{m})(x_0) = g_{m}(x_0)$, $F(\bullet_A)(x_0)= Ax_0$,
			\item If $\tau_1, \tau_2, \ldots, \tau_\kappa\in T$, then $F([\tau_1]_A)(x_0)= AF(\tau_1)(x_0)$ and
			\[
			F([\tau_1, \tau_2, \ldots, \tau_\kappa]_{m})(x_0) = g_{m}^{(\kappa)}(x_0)(F(\tau_1)(x_0), \ldots, F(\tau_\kappa)(x_0))
			\]
		\end{enumerate}
		for $m=0,\dots,M$.
	\end{definition}
	Now we give the definition of B-series.
	\begin{definition}[B-series] A (stochastic) B-series is a formal series of the form
		\[
		B(\phi, x_0; h) =\sum_{\tau\in T}\alpha(\tau)\cdot\phi(\tau)(h)\cdot F(\tau)(x_0),
		\]
		where $\phi(\tau)(h)$ is a random variable satisfying $\phi(\emptyset)\equiv1$  and $\phi(\tau)(0)=0$ for all $\tau\in T\setminus\{\emptyset\}$, and
		$\alpha : T\to\mathbb{Q}$ is given by
		\begin{align*}
		& \alpha(\emptyset) = 1, \qquad \alpha(\bullet_{m}) = 1, \qquad \alpha(\bullet_A) = 1, \\
		&\alpha([\tau_1, \ldots, \tau_\kappa]_{m}) = \frac{1}{r_1!r_2!\cdots r_{l}!}\prod_{{k}=1}^{\kappa}\alpha(\tau_{k}), \quad \alpha([\tau_1]_A) = \alpha(\tau_1),
		\end{align*}
		where $r_1,r_2,\ldots, r_{l}$ count equal trees among $\tau_1, \tau_2, \ldots, \tau_\kappa$, and $m=0,\dots,M$.
	\end{definition}
	
	Next we give an important lemma to derive B-series for the exact and numerical solutions. It
	states that if $Y(h)$ can be expressed as a B-series, then $f(Y(h))$ can also be expressed
	as a B-series where the sum is taken over trees with a root of color $f$ and subtrees in $T$.
	
	\begin{lemma}\label{kvarno_mini_ms18:lemma:fBseries}
		If $Y(h)= B(\phi,x_0;h)$ is some B-series and $f\in C^\infty(\mathbb{R}^d, \mathbb{R}^d)$,then $f(Y(h))$ can be written as a formal series of the form
		\begin{equation}\label{kvarno_mini_ms18:eq:fBseries}
		f(Y(h)) = \sum_{u\in U_f}\beta(u)\cdot\psi_\phi(u)(h)\cdot G(u)(x_0)
		\end{equation}
		where $U_f$ is a set of trees derived from $T$, by
		\begin{enumerate}
			\item[(i)] $[\emptyset]_f\in U$, and if $\tau_1, \tau_2, \ldots, \tau_\kappa\in T$, then $u=[\tau_1, \tau_2, \ldots, \tau_\kappa]_f\in U_f.$
			\item[(ii)] $G([\emptyset]_f)(x_0) = f(x_0)$ and
			\[
			G([\tau_1, \tau_2, \ldots, \tau_\kappa]_f)(x_0) = f^{(\kappa)}(x_0)(F(\tau_1)(x_0),\ldots, F(\tau_\kappa)(x_0)).
			\]
			\item[(iii)] $\beta([\emptyset]_f) = 1$ and $\beta([\tau_1, \ldots, \tau_\kappa]_f) = \frac{1}{r_1!r_2!\cdots r_{l}!}\prod_{{k}=1}^{\kappa}\alpha(\tau_{k})$, with $r_1,$ $r_2,$ $\ldots,$ $r_{l}$ counting equal trees among $\tau_1,\tau_2,\ldots,\tau_\kappa$.
			\item[(iv)] $\psi_\phi([\emptyset]_f)\equiv1$ and $\psi_\phi([\tau_1, \tau_2, \ldots, \tau_\kappa]_f)(h) = \prod_{{k}=1}^{\kappa}\phi(\tau_{k})(h)$.
		\end{enumerate}
	\end{lemma}
	\begin{proof}
		The proof of this lemma is given in \cite{kvarno_mini_ms18:debrabant08bsa}.\qed
	\end{proof}
	Applying Lemma \ref{kvarno_mini_ms18:lemma:fBseries} to the functions $g_{m}$ on the right hand side of \eqref{kvarno_mini_ms18:eq:SDE} gives
	\begin{equation}\label{kvarno_mini_ms18:eq:gmBseries}
	g_{m}(B(\phi, x_0;h)) = \sum_{\tau\in T_{m}}\alpha(\tau)\cdot\phi_{m}'(\tau)(h)\cdot F(\tau)(x_0),
	\end{equation}
	where
	\[
	\phi_{m}'(\tau)(h) = \begin{cases}
	1 & \text{ if } \tau = \bullet_{m},\\ \prod_{{k}=1}^{\kappa}\phi(\tau_{k})(h) & \text{ if } \tau = [\tau_1,\ldots,\tau_\kappa]_{m}\in T_{m}.
	\end{cases}
	\]
	Assume the exact solution $X(h)$ of \eqref{kvarno_mini_ms18:eq:SDE} at $t = h$ can be written as a B-series $B(\varphi, x_0; h)$.
	Substituting $X(h)=B(\varphi, x_0; h)$ in \eqref{kvarno_mini_ms18:eq:SDEexpform} and using \eqref{kvarno_mini_ms18:eq:gmBseries} gives
	\begin{multline*}
	B(\varphi, x_0; h)  \\ = \mathrm{e}^{hA}x_0
	+\sum_{{m}=0}^{{M}}\int_{0}^{h}\mathrm{e}^{(h-s)A}\sum_{\hat{\tau}\in
      T_{m}}\alpha(\hat{\tau})\cdot\varphi_{m}'(\hat{\tau})(s)\cdot F(\hat{\tau})(x_0)\star\mathrm{d} W_{m}(s).
	\end{multline*}
	Inserting the series representation $\mathrm{e}^{hA}x_0 = \sum_{{q}=0}^{\infty}\frac{h^{q} A^{q}}{{q}!}x_0$ yields
	\begin{multline}\label{kvarno_mini_ms18:eq:Bseriesimplicit}
	B(\varphi, x_0; h)
	= x_0 + \sum_{{q}=1}^{\infty}\frac{h^{q}}{{q}!}A^{{q}}x_0 \\ +
	\sum_{{m} = 0}^{{M}}\sum_{\hat{\tau}\in T_{m}}\alpha(\hat{\tau})\sum_{{q}=0}^{\infty}
	\bigg(\int_{0}^{h}\frac{(h-s)^{q}}{{q}!}\varphi_{m}'(\hat{\tau})(s)\star\mathrm{d} W_{m}(s)\cdot A^{q} F(\hat{\tau})(x_0) \bigg).
	\end{multline}
	Note that any tree $\tau\in T$ can be rewritten as $\tau =
	[\ldots[[\hat{\tau}\overbrace{]_A]_A\ldots]_{A}}^{{q}\text{-times}}=[\hat{\tau}]_A^{q}$
	for ${q} = 0, 1, \ldots$, with  $\hat{\tau}\in T\setminus T_A$, that means $\hat{\tau}=\emptyset$ or
	$\hat{\tau}=[\tau_1,\ldots,\tau_\kappa]_{m}$ for an ${m}\in\{1,\dots,{M}\}$. It holds that
	$F([\hat{\tau}]_A^{q})=A^{q} F(\hat{\tau})$,
	$\alpha([\hat{\tau}]_A^{q})=\alpha(\hat{\tau})$. Especially, for $\tau =
	[\emptyset]_A^{q}$ it holds that $\alpha(\tau) =1$ and $F(\tau)(x_0) = A^{q}
	F(\emptyset)=A^{q} x_0$. In conclusion \eqref{kvarno_mini_ms18:eq:Bseriesimplicit} implies the following theorem:
	\begin{theorem}\label{kvarno_mini_ms18:thm:exact}
		The solution $X(h)$ of the SDE \eqref{kvarno_mini_ms18:eq:SDE} can be written as a B-series $B(\varphi, x_0; h)$ with
		\begin{gather*}
		\varphi(\emptyset)(h) = 1,\qquad \varphi([\emptyset]_A^{q})(h) = \frac{h^{q}}{{q}!},
\\
		\varphi([[\tau_1,\ldots,\tau_\kappa]_{m}]_A^{q})(h) =
		\int_{0}^{h}\frac{(h-s)^{q}}{{q}!}\prod_{{k}=1}^{\kappa}\varphi(\tau_{k})(s)\star\mathrm{d} W_{m}(s),
		\end{gather*}
		for $\tau_1,\dots,\tau_\kappa\in T$, $\kappa=0,1,\ldots$, ${q}=0,1,\ldots$ and $m=0,\dots,M$, where $\tau_i\neq\emptyset$ for $i=1,\dots,\kappa$ if $\kappa>1$.
	\end{theorem}
	\begin{example}\label{kvarno_mini_ms18:ex:treexample}
		Let $\tau = \tikz[grow=up, level distance=2mm,	sibling distance=2mm, every node/.style={shading=ball, ball color=black, circle, inner sep=0.5mm}]{\node {} child{node[ball color=white] {} child{node{}} child{node[ball color=red] {}}} child{node[ball color=white] {}}}$, where the
		colors red, black and white correspond to the matrix $A$, deterministic function
		$g_0$ and stochastic function $g_1$ respectively. Then $\alpha(\tau) = 1$,
		$F(\tau)(x_0) = g_0''(g_1,g_1''(Ax_0, g_0))(x_0)$ and $\varphi(\tau)(h) =
		\int_{0}^{h}\left(W_1(s) \int_{0}^{s}s_1^2\star\mathrm{d} W_1(s_1)\right) \mathrm{d} s$. Note also that
		e.g.\ the tree $\tau = \tikz[grow=up, level distance=2mm,	sibling distance=2mm, every node/.style={shading=ball, ball color=black, circle, inner sep=0.5mm}]{\node {} child{node[ball color=red] {} child{node{}} child{node[ball color=red] {}}}
		child{node[ball color=white] {}}} \notin T$ since it is impossible for node $\tikz[grow=up, level distance=2mm,	sibling distance=2mm, every node/.style={shading=ball, ball color=black, circle, inner sep=0.5mm}]{\node[ball color=red] {}}$ to have more than one branch.
	\end{example}
	
	Now we derive the B-series representation for one step of the exponential stochastic
	integrator \eqref{kvarno_mini_ms18:eq:Btab}. Assume both the stage values $H_i$ and the
	approximation $Y_{n+1}$ to the exact solution  can be written as B-series:
	\begin{equation}\label{kvarno_mini_ms18:BSeriesansatznum}
	H_i = B(\Phi_i,Y_n;h), \qquad i = 1,\ldots, \nu \qquad \text{ and } \qquad Y_{n+1} = B(\Phi, Y_n; h).
	\end{equation}
	In addition we assume that the coefficients $Z_{ij}^{(m)}(A)$ and $z_i^{(m)}(A)$ can be expressed as power series of the form
	\begin{equation*}
	Z_{ij}^{({m})}(A) = \sum_{{q}=0}^{\infty}Z_{ij}^{({m},{q})}A^{q} \qquad \text{ and }\qquad z_{i}^{({m})}(A) = \sum_{{q}=0}^{\infty}z_{i}^{({m},{q})}A^{q},
	\end{equation*}
	for $i,j = 1, \ldots, \nu,$ and ${m} = 0,\ldots, {M}$.
	Substituting these formulas into \eqref{kvarno_mini_ms18:eq:Btab} and using \eqref{kvarno_mini_ms18:eq:gmBseries} we get
	\begin{align*}
	&H_i  = \sum_{{q}=0}^{\infty}\frac{(c_ih)^{{q}}}{{q}!}A^{{q}}Y_n +
	\sum_{{m}=0}^{{M}}\sum_{j=1}^{\nu}Z_{ij}^{({m})}(A)\sum_{\tau\in T_{m}}\alpha(\tau)\cdot\Phi_j'(\tau)(h)\cdot F(\tau)(Y_n) \\
	& = \sum_{{q}=0}^{\infty}\frac{(c_ih)^{{q}}}{{q}!}A^{{q}}Y_n
	+ \sum_{{m}=0}^{{M}}\sum_{j=1}^{\nu}\sum_{\tau\in T_{m}}\alpha(\tau)\sum_{{q}=0}^{\infty}Z_{ij}^{({m},{q})}\Phi_j'(\tau)(h)\cdot A^{q}  F(\tau)(Y_n),
	\end{align*}
	and similarly
	\begin{align*}
	Y_{n+1} & = \sum_{{q}=0}^{\infty}\frac{h^{{q}}}{{q}!}A^{{q}}Y_n
	+ \sum_{{m}=0}^{{M}}\sum_{i=1}^{\nu}\sum_{\tau\in T_{m}}\alpha(\tau)\sum_{{q}=0}^{\infty}z_{i}^{({m},{q})}\Phi_i'(\tau)(h)\cdot A^{q}  F(\tau)(Y_n).
	\end{align*}
	Now using \eqref{kvarno_mini_ms18:BSeriesansatznum} and the linear independence of the elementary differentials yields the following theorem.
	\begin{theorem}\label{kvarno_mini_ms18:thm:SETDRK}
		The stage values $H_i$ and the numerical solution $Y_{n+1}$ of \eqref{kvarno_mini_ms18:eq:Btab} can be written as B-series $H_i = B(\Phi_i,Y_n;h)$, $i = 1,\ldots, \nu$, and $Y_{n+1} = B(\Phi, Y_n; h)$ with the following recurrence relations for the functions $\Phi_i(\tau)(h)$ and $\Phi(\tau)(h)$,
		\begin{gather*}
		\Phi_i(\emptyset) = \Phi(\emptyset)\equiv1,\quad \Phi_i([\emptyset]_A^{q})(h) =\frac{(c_ih)^{{q}}}{{q}!},\quad \Phi([\emptyset]_A^{q})(h) = \frac{h^{q}}{{q}!},\\
		\Phi_i([[\tau_1,\ldots,\tau_\kappa]_{m}]_A^{q})(h) = \sum_{j=1}^{\nu}Z_{ij}^{({m},
		{q})}\prod_{{k}=1}^{\kappa}\Phi_j(\tau_{k})(h),\\
		\Phi([[\tau_1,\ldots,\tau_\kappa]_{m}]_A^{q})(h) = \sum_{i=1}^{\nu}z_{i}^{({m}, {q})}\prod_{{k}=1}^{\kappa}\Phi_i(\tau_{k})(h),
		\end{gather*}
		for $\tau_1,\dots,\tau_\kappa\in T$, $\kappa=0,1,\ldots$, ${q}=0,1,\ldots$  and $m=0,\dots,M$, where $\tau_i\neq\emptyset$ for $i=1,\dots,\kappa$ if $\kappa>1$.
	\end{theorem}
	To discuss the order of the method, we need the following definition.
	\begin{definition}[order]
		The order $\rho(\tau)$ of a tree $\tau\in T$ is defined by
		\[
		\rho(\emptyset) = 0, \qquad \rho([\tau_1]_A) = \rho(\tau_1) +1
		\]
		and
		\[\rho([\tau_1,\ldots,\tau_\kappa]_{m} ) =\sum_{{k}=1}^{\kappa}\rho(\tau_{k}) + \begin{cases}
		1 & \text{ if } {m}=0, \\ \frac{1}{2} & \text{otherwise,}
		\end{cases}
		\]
		for ${m} = 0, 1, \ldots, M$.
	\end{definition}
With Theorems \ref{kvarno_mini_ms18:thm:exact} and \ref{kvarno_mini_ms18:thm:SETDRK} in place, we can now analyze the order of a given method:
	\begin{theorem} \label{kvarno_mini_ms18:thm:milstein}
		The method has mean square global order $p$ if
		\begin{subequations} \label{kvarno_mini_ms18:eq:msord}
			\begin{align}
			\Phi(\tau)(h) &= \varphi(\tau)(h) +
			\mathcal{O}(h^{p+\frac{1}{2}}) \text{ for all }\tau  \in T \text{ with }
			\rho(\tau)\leq p, \label{kvarno_mini_ms18:eq:msorda} \\
			E\Phi(\tau)(h) &= E \varphi(\tau)(h) + \mathcal{O}(h^{p+1}) \text{ for all }\tau
			\in T \text{ with }
			\rho(\tau)\leq p+\frac{1}{2}. \label{kvarno_mini_ms18:eq:msordb}
			\end{align}
		\end{subequations}
	\end{theorem}
	Here, the $\mathcal{O}(\cdot)$-notation refers to $h\to0$ and, especially in \eqref{kvarno_mini_ms18:eq:msorda}, to the $L^2$-norm.
	The result \eqref{kvarno_mini_ms18:eq:msord} was first proved in \cite{kvarno_mini_ms18:burrage00oco}.\\
We conclude this article with an example.
	\begin{example}
		We will apply Theorem \ref{kvarno_mini_ms18:thm:milstein} to the method given in Example \ref{kvarno_mini_ms18:ex:SETDRK}. Using the expansion (for the manipulation of stochastic integrals, see e.\,g.\ \cite{kvarno_mini_ms18:debrabant10ste})
		\begin{align*}
		 \int_{0}^{h}\mathrm{e}^{(h-s)A}\star\mathrm{d} W_1(s) & = \int_{0}^{h}1\star\mathrm{d} W_1(s)A^0 +
		 \int_{0}^{h}(h-s)\star\mathrm{d} W_1(s)A^1 \\
		 & +\int_{0}^{h}\frac{(h-s)^2}{2}\star\mathrm{d} W_1(s)A^2+\dots \\
		 &= I_{(1)}^{*}A^0 + I_{(10)}^{*} A^1 + I_{(100)}^{*}A^2 +\dots\\
		\end{align*}
		 where $I_{(m_1\dots m_n)}^*=\int_{0}^{h}\int_{0}^{s_1}\dots \int_{0}^{s_{n-1}}\star\mathrm{d} W_{m_1}(s_n)\dots\star\mathrm{d} W_{m_n}(s_1)$, and the similar expansion $\int_{0}^{h}\mathrm{e}^{(h-s)A}W_1(s)\star\mathrm{d} W_1(s)= I_{(11)}^{*}A^0 + I_{(110)}^{*}A^1 + \dots$
		 we obtain
		\begin{align*}
		z_1^{(0)}&=\int_0^h\mathrm{e}^{(h-s)A} \mathrm{d} s=
		hA^0 + \frac{h^2}{2}A^1 + \frac{h^3}{6}A^2 + \dots,
		\\
		z_1^{(1)}&= \int_{0}^{h}\mathrm{e}^{(h-s)A}(1 -\frac{W_1(s)}{\sqrt{h}})\star\mathrm{d} W_1(s)\\
			& = (I_1^*-\frac{I_{(11)}^{*}}{\sqrt{h}})A^0 + (I_{(10)}^*-\frac{I_{(110)}^{*}}{\sqrt{h}})A^1 + \dots,\\
		z_2^{(1)}&=\int_{0}^{h}\mathrm{e}^{(h-s)A}\frac{W_1(s)}{\sqrt{h}}\star\mathrm{d} W_1(s) = \frac{I_{(11)}^{*}}{\sqrt{h}}A^0 + \frac{I_{(110)}^{*}}{\sqrt{h}}A^1 + \dots.
		\end{align*}
		We also have (with colors as in Example \ref{kvarno_mini_ms18:ex:treexample}) $z_2^{(0)}=0$, $\Phi_1(\tikz[grow=up, level distance=2mm,	sibling distance=2mm, every node/.style={shading=ball, ball color=black, circle, inner sep=0.5mm}]{\node {}})=\Phi_2(\tikz[grow=up, level distance=2mm,	sibling distance=2mm, every node/.style={shading=ball, ball color=black, circle, inner sep=0.5mm}]{\node {}})=\Phi_1(\tikz[grow=up, level distance=2mm,	sibling distance=2mm, every node/.style={shading=ball, ball color=black, circle, inner sep=0.5mm}]{\node[ball color=red] {}})=\Phi_2(\tikz[grow=up, level distance=2mm,	sibling distance=2mm, every node/.style={shading=ball, ball color=black, circle, inner sep=0.5mm}]{\node[ball color=red] {}})=\Phi_1(\tikz[grow=up, level distance=2mm,	sibling distance=2mm, every node/.style={shading=ball, ball color=black, circle, inner sep=0.5mm}]{\node[ball color=white] {}})=\Phi_1(\tikz[grow=up, level distance=2mm,	sibling distance=2mm, every node/.style={shading=ball, ball color=black, circle, inner sep=0.5mm}]{\node[ball color=white] {} child {node[ball color=white] {}}})=\Phi_2(\tikz[grow=up, level distance=2mm,	sibling distance=2mm, every node/.style={shading=ball, ball color=black, circle, inner sep=0.5mm}]{\node[ball color=white] {} child {node[ball color=white] {}}})=0$ and $\Phi_2(\tikz[grow=up, level distance=2mm,	sibling distance=2mm, every node/.style={shading=ball, ball color=black, circle, inner sep=0.5mm}]{\node[ball color=white] {}})=\sqrt{h}$, resulting in the weight functions given in 
		the following table:\\[1mm]{\small
			\begin{tabular}{c|c|c|c}\hline
				$\tau$         & $\rho(\tau)$ & $\varphi(\tau)(h)$ & $\Phi(\tau)(h)$\\ \hline
				\tikz[grow=up, level distance=2mm,	sibling distance=2mm, every node/.style={shading=ball, ball color=black, circle, inner sep=0.5mm}]{\node[ball color=white] {}} & 0.5 & $I_{(1)}^*$ &  $z_1^{(1,0)} + z_2^{(1,0)} = I_{(1)}^*$ \\ \hline
				\tikz[grow=up, level distance=2mm,	sibling distance=2mm, every node/.style={shading=ball, ball color=black, circle, inner sep=0.5mm}]{\node {}} &     1        & $h$             & $z_1^{(0,0)} + z_2^{(0,0)} = h$\\
				\tikz[grow=up, level distance=2mm,	sibling distance=2mm, every node/.style={shading=ball, ball color=black, circle, inner sep=0.5mm}]{\node[ball color=red] {}} &            & $h$             & $h$ \\
				\tikz[grow=up, level distance=2mm,	sibling distance=2mm, every node/.style={shading=ball, ball color=black, circle, inner sep=0.5mm}]{\node[ball color=white] {} child{node[ball color=white] {}}} &    & $I_{(11)}^*$ & $z_1^{(1,0)}\Phi_1(\tikz[grow=up, level distance=2mm,	sibling distance=2mm, every node/.style={shading=ball, ball color=black, circle, inner sep=0.5mm}]{\node[ball color=white] {}}) + z_2^{(1,0)}\Phi_2(\tikz[grow=up, level distance=2mm,	sibling distance=2mm, every node/.style={shading=ball, ball color=black, circle, inner sep=0.5mm}]{\node[ball color=white] {}}) = I_{(11)}^*$  \\
				\hline
				\tikz[grow=up, level distance=2mm,	sibling distance=2mm, every node/.style={shading=ball, ball color=black, circle, inner sep=0.5mm}]{\node {} child{node[ball color=white] {}}} & 1.5   & $I_{(10)}^*$ & $z_1^{(0,0)}\Phi_1(\tikz[grow=up, level distance=2mm,	sibling distance=2mm, every node/.style={shading=ball, ball color=black, circle, inner sep=0.5mm}]{\node[ball color=white] {}}) + z_2^{(0,0)}\Phi_2(\tikz[grow=up, level distance=2mm,	sibling distance=2mm, every node/.style={shading=ball, ball color=black, circle, inner sep=0.5mm}]{\node[ball color=white] {}}) = 0$ \\
				\tikz[grow=up, level distance=2mm,	sibling distance=2mm, every node/.style={shading=ball, ball color=black, circle, inner sep=0.5mm}]{\node[ball color=red] {} child{node[ball color=white] {}}} &       & $hI_{(1)}^*-I_{(01)}^*$ & $z_1^{(1,1)}\Phi_1(\emptyset) + z_2^{(1,1)}\Phi_2(\emptyset) = I_{(10)}^*$\\
				\tikz[grow=up, level distance=2mm,	sibling distance=2mm, every node/.style={shading=ball, ball color=black, circle, inner sep=0.5mm}]{\node[ball color=white] {} child{node{}}} &       & $I_{(01)}^*$ & $z_1^{(1,0)}\Phi_1(\tikz[grow=up, level distance=2mm,	sibling distance=2mm, every node/.style={shading=ball, ball color=black, circle, inner sep=0.5mm}]{\node {}}) + z_2^{(1,0)}\Phi_2(\tikz[grow=up, level distance=2mm,	sibling distance=2mm, every node/.style={shading=ball, ball color=black, circle, inner sep=0.5mm}]{\node {}}) = 0$\\
				\tikz[grow=up, level distance=2mm,	sibling distance=2mm, every node/.style={shading=ball, ball color=black, circle, inner sep=0.5mm}]{\node[ball color=white] {} child{node[ball color=red] {}}} &       & $I_{(01)}^*$ & $z_1^{(1,0)}\Phi_1(\tikz[grow=up, level distance=2mm,	sibling distance=2mm, every node/.style={shading=ball, ball color=black, circle, inner sep=0.5mm}]{\node[ball color=red] {}}) + z_2^{(1,0)}\Phi_2(\tikz[grow=up, level distance=2mm,	sibling distance=2mm, every node/.style={shading=ball, ball color=black, circle, inner sep=0.5mm}]{\node[ball color=red] {}}) = 0$\\
				\tikz[grow=up, level distance=2mm,	sibling distance=2mm, every node/.style={shading=ball, ball color=black, circle, inner sep=0.5mm}]{\node[ball color=white] {} child{node[ball color=white] {}} child{node[ball color=white] {}}} &       & $\int_{0}^{h}W_1^2(s)\star\mathrm{d} W_1(s)$ & $z_1^{(1,0)}\Phi_1^2(\tikz[grow=up, level distance=2mm,	sibling distance=2mm, every node/.style={shading=ball, ball color=black, circle, inner sep=0.5mm}]{\node[ball color=white] {} }) + z_2^{(1,0)}\Phi_2^2(\tikz[grow=up, level distance=2mm,	sibling distance=2mm, every node/.style={shading=ball, ball color=black, circle, inner sep=0.5mm}]{\node[ball color=white] {} }) = \sqrt{h}I_{(11)}^*$\\
				\tikz[grow=up, level distance=2mm,	sibling distance=2mm, every node/.style={shading=ball, ball color=black, circle, inner sep=0.5mm}]{\node[ball color=white] {} child{node[ball color=white] {} child{node[ball color=white] {}}}} &       & $I_{(111)}^*$ & $z_1^{(1,0)}\Phi_1(\tikz[grow=up, level distance=2mm,	sibling distance=2mm, every node/.style={shading=ball, ball color=black, circle, inner sep=0.5mm}]{\node[ball color=white] {} child {node[ball color=white] {}}}) + z_2^{(1,0)}\Phi_2(\tikz[grow=up, level distance=2mm,	sibling distance=2mm, every node/.style={shading=ball, ball color=black, circle, inner sep=0.5mm}]{\node[ball color=white] {} child{node[ball color=white] {}}}) = 0$ \\ \hline
			\end{tabular}}\\[1mm]
			While the weight functions for the exact solution and numerical approximation of the order 1.5 trees do not coincide, their expectation values coincide for It\^{o} integral but not for Stratonovich (when $\tau = \tikz[grow=up, level distance=2mm,	sibling distance=2mm, every node/.style={shading=ball, ball color=black, circle, inner sep=0.5mm}]{\node[ball color=white] {} child{node[ball color=white] {}} child{node[ball color=white] {}}}$). Thus, the method given in Example \ref{kvarno_mini_ms18:ex:SETDRK} has mean square order 1 for the It\^{o} case but 0.5 in the Stratonovich case.
		\end{example}	
		\bibliographystyle{vmams}	
		   \ifx\undefined\bysame
		   \newcommand{\bysame}{\leavevmode\hbox to3em{\hrulefill}\,}
		   \fi

	\end{document}